\documentclass{article}
\usepackage{graphicx}
\usepackage{amsthm}
\usepackage{latexsym}
\usepackage{multirow}
\usepackage{amsmath}
\usepackage{epsfig}
\usepackage{color}
\usepackage{amsfonts}
\usepackage{algorithm}
\theoremstyle{plain}
\newtheorem{theorem}{Theorem}[section]
\newtheorem{lemma}[theorem]{Lemma}

\newtheorem{proposition}[theorem]{Proposition}

\theoremstyle{definition}
\newtheorem{definition}[theorem]{Definition}
\newtheorem{example}[theorem]{Example}

\theoremstyle{remark}

\newcommand{\R}{\mathbb{R}}
\newcommand{\N}{\mathbb{N}}

\begin{document}


\title{A parallel subgradient projection  algorithm\\ for quasiconvex equilibrium problems under the  intersection  of convex sets\thanks{This work is supported by the NAFOSTED, Grant 101.01-2020.06.}}

\author{Le Hai Yen and Le Dung Muu
\\ Institute of Mathematics and TIMAS, Thang Long University}

\date{Email. lhyen@math.ac.vn; ldmuu@math.ac.vn}
\maketitle

\begin{abstract}
In this paper, we studied the equilibrium problem where the bi-function may be quasiconvex with respect to the second variable and the feasible set is the intersection of a finite number of convex sets. We propose a projection-algorithm, where the projection can be computed independently onto each component set. The convergence of the algorithm is investigated and numerical examples for a variational inequality problem involving affine fractional operator are provided to demonstrate the behavior of the algorithm.
\end{abstract}


\font\abst=cmr12

\section{Introduction}Let $C$ be a nonempty closed convex set in $\R^n$ and $f:  \R^n \to \R \cup \{+\infty\}$ be a given bifunction such that $f(x,y) < + \infty$ for every $x, y \in C$. We consider the problem
$$ \textnormal{Find } \ x^*\in C:\  f(x^*,y)  \geq 0 \quad \forall y\in C. \eqno (EP)$$
This inequality   is often called equilibrium problem.
 The interest of this
problem is that it unifies many important problems such as the Kakutani fixed
point, variational inequality, optimization and the Nash equilibrium problems \cite{Bi2013,Bi2018,Bl1994,Mu1992} in a
convenient way. The inequality in (EP) first was used in \cite{Ni1955} by Nikaido and Isoda for a convex game
model. The first result for solution existence of (EP) has been obtained by
Ky Fan in \cite{Fa1972}, where the bifunction $f$ can be quasiconvex with respect to the
second argument.
 Suppose, as usual, that $f(x,x) = 0$ for all $x\in C$, then it is easy to see that (EP) is equivalent to the fixed point problem
$$ \textnormal{Find } x^* \in C: x^* \in S(x^*), \eqno(FP)$$
where the fixed point mapping $S$ is defined by taking
$$S(x) := \textnormal{argmin} \{ f(x,y): y\in C\}. \eqno P(x)$$
 This fact suggests the use of the iterative scheme $x^{k+1} \in S(x^k)$ for the fixed point problem to solve inequality $(EP)$. The first difficulty that we have to face with here is that the mapping $S$ may  not be singleton, even it is  not defined at every point of $C$, i.e., Problem $P(x)$ is not solvable. To overcome this difficulty   one can use the auxiliary problem principle that states that if $f(x,.)$ is convex, subdifferentiable on $C$, then for any $r>0$, Problem $(EP)$ is equivalent to the following one
$$ \textnormal{Find } \ x^*\in C:\  f(x^*,y) + r\|y-x^*\|^2 \geq 0 \quad \forall y\in C, \eqno (EP) $$
in the sense that their solution-sets coincide.
Then the corresponding  fixed point mapping  takes the form
\begin{equation}\label{1.1}
s(x) := \textnormal{argmin} \{ f(x,y) + r\|y-x\|^2: y\in C\ \}.
\end{equation}
Since $f(x,.)$ is convex and $\|.-x\|^2$ is strongly convex, the latter mathematical program is always uniquely solvable.
However  in the case $f(x,.)$ is quasiconvex rather than convex, the auxiliary problem principle cannot be applied because of the fact that the mathematical programming problem defining the mapping $s(.)$ is nonlonger convex, even not quasiconvex, and therefore solving it is an extremely difficult task. Based upon the auxiliary problem principle, a lot numbers of algorithms    using techniques of mathematical programming methods have been developed for solving problem (EP), e.g. \cite{Bi2018,Hi2019,MQ2009,QMH2008,Sa2011,San2016,Th2017,Yen2019} and the references therein, however to our best knowledges, all of them require that the bifunction $f$ is convex with respect to its second variable. In our recent preprint \cite{Ye2020}  we developed  an  algorithm for problem (EP), where the bifunction $f$ may be quasiconvex in its second variable. In order to handle the quasiconvexity, we used the metric projection onto $C$ along the direction defined by a star-subgradient of the quasiconvex function $f(x,.)$, rather than solving the mathematical programming problem (\ref{1.1}). However in  general, the projection onto $C$ is not easy to compute. In this paper we continue our work by considering problem (EP) where the feasible set $C$ is the intersection of a finite number of convex sets (often in practice), and we propose a projection-algorithm, where the projection can be computed independently onto each component set.

The organization of this paper as the following. The next section are preliminaries on the quasiconvex function on $\R^n$ and its star-subdifferential. A parallel algorithm  for solving (EP) when $C$ is the intersection of a finite number of convex sets is proposed,  and its convergence analysis is studied in Section 3. In the last section, numerical experiences are provided to prove the efficiency of the algorithm for a class of equilibrium problems involving quasiconvex bifunctions.

             \section{ Preliminaries on Quasiconvex Function and Its Subdifferentials}
First of all, let us recall the well known definitions on the quasiconvex function and its star- subdifferential that will be used to the algorithm.

\begin{definition}\cite{Ma1969}
A function $\varphi:\R^n \rightarrow \R \cup \{+\infty\}$ is called quasiconvex on a convex subset $Y$ of $ \R^n$ if and only if for every $x,y \in Y$ and $\lambda \in [0,1]$, one has
\begin{equation}
\varphi[(1-\lambda)x +\lambda y] \leq \max[\varphi(x), \varphi(y)].\label{eq1}
\end{equation}
\end{definition}
It is easy to see that $\varphi$ is quasiconvex on a convex set  $Y$   if and only if the level set $ \{ x\in Y: \quad \varphi(x)<\alpha\}$  on $Y$ of $\varphi$ at $x$ is convex for every $\alpha \in \R$.

We recall  that a function $\varphi: \R^n \to \R$ is said to be  Lipschitz on $Y$ at a point $y \in Y$, if there exist a finite number $ L > 0$  such that
$$ |\varphi (x) - \varphi (y)| \leq L \|x-y\| \  \forall x\in Y.$$

The star-sudifferential of $\varphi$, see e,g. \cite{Pe1998}  is defined as
  $$ \partial^*\varphi(x) :=\{g\in \R^n:\  \langle g, y-x\rangle < 0 \ \forall y\in L_\varphi(x)\},$$
where $L_\varphi(x) := \{y\in R^n: \ \varphi(y) < \varphi (x)\}$ is the level set of $\varphi$ at the  level $\varphi(x)$. Clearly, if $\bar{L}_\varphi(x) $ is the closure of $L_\varphi(x)$, then
$$ \partial^*\varphi(x) :=\{g\in \R^n;\  \langle g, y-x\rangle \leq 0 \ \forall y\in \bar{L}_\varphi(x)\}.$$
Hence $\partial^* \varphi (x) \equiv \R^n$ if $x$ is a minimizer of $\varphi$ over $\R^n$, and if $\varphi$ is continuous on $\R^n$ then
$\partial^*\varphi(x)$ is the normal cone of $\bar{L}_\varphi(x)$, that is
$$\partial^* \varphi(x) =  N (\bar{L}_\varphi(x),x):= \{g\in \R^n: \ \langle g, y-x\rangle \leq 0\  \forall y \in \bar{L}_\varphi(x)\}.$$
Furthermore $\partial^*\varphi(x)$ contains nonzero vector \cite{GP1973}. This subdifferential thus is also called normal-subdifferential.


\begin{lemma}(\cite{Ki2001}, \cite{Pe2000})
\label{lem1} Assume that $\varphi:\R^n \rightarrow \R $ is continuous and quasiconvex. Then
\begin{equation}
\partial^* \varphi(x) \not= \emptyset \quad \forall x\in \R^n,
\end{equation}

 \begin{equation}
0\in \partial^* \varphi(x)  \Leftrightarrow x \in \textnormal{argmin} \{ \varphi(y): \ y\in \R^n \}.
\end{equation}
\end{lemma}

 For simplicity of notation, let $f_k(x) := f(x^k,x)$. For the star-subdifferential we have the following results will be used in the sequel.

\begin{lemma}{\cite{Ki2001}}
\label{lem3.1a}
 If $B(\overline{x}, \epsilon) \subset L_{f_k}(x^k)$ for some $\overline{x}\in \R^n$ and $\epsilon \geq 0$, then
\begin{equation}
\langle g^k, x^k-\overline{x}\rangle >\epsilon.\nonumber
\end{equation}
\end{lemma}

\begin{lemma}\label{Ko1}
Let $h: \mathbb{R}^n \longrightarrow \mathbb{R}$ be a quasiconvex $L$-Lipschitz continuous function on $\mathbb{R}^n$. Then for every $y\in \overline{L}_{h}(z)$, $g \in N(\bar{L}_h(z), z)$ such that $\|g\|=1$, it holds
\begin{equation}
L \langle g, y-z \rangle \leq h(y) -h(z) \leq 0,\label{eqlem}
\end{equation}


\end{lemma}
\begin{proof}
Let $y\in \overline{L}_{h}(z)$. Thanks to the continuity of the function $h$, we have   $h(y)-h(z)\leq 0$.
Since $z\not\in L_h(z)$,  $N(\bar{L}_h(z), z ) \not=\emptyset$.  Take $g \in N(\bar{L}_h(z), z )$ such that $\|g\|=1$, then
$$\langle g, x-z \rangle \leq 0 \quad \forall x \in \bar{L}_h(z).$$
Let $H(g)$ be the supporting hyperplane of $ \bar{L}_h(z) $ at $z$, that is defined by
$$H(g)=\{x\in \mathbb{R}^n| \quad \langle g, x-z\rangle =0\}.$$
Since $L_h(z)$ is open, $H(g) \cap L_h(z) =\emptyset$.
If $y\in H(g)$ then it is clear that (\ref{eqlem}) is true. If $y\not \in H(g)$,
let $P(y)$ be the projection of $y$ onto the hyperplane $H(g)$. Then     $h(P(y))\geq h(z)$.

On one hand,

\begin{eqnarray}
0\leq h(z)-h(y)& \leq & h(P(y))-h(y)\nonumber\\
&\leq & L\|P(y)-y\|.\label{eqlem1}
\end{eqnarray}

On the other hand,
\begin{eqnarray}
\|P(y)-y\|&=& \langle g , P(y)-y\rangle \nonumber\\
&=& \langle g, P(y)-z\rangle +\langle g, z-y\rangle \nonumber\\
&=& \langle g, z-y\rangle (\textnormal{because } P(y)\in H(g))\label{eqlem2}
\end{eqnarray}
From (\ref{eqlem1}) and (\ref{eqlem2}),
\begin{equation}
L \langle g, y-z \rangle \leq h(y) -h(z) \leq 0.\nonumber
\end{equation}
\end{proof}







We also need the following result to find the subdifferential   in some important cases. Some calculus rules, optimality conditions and minimization methods concerning   these subdifferentials have been studied in \cite{De1998,GP1973,GRo1998,Ju1998,Ki2001,Ko2003,Pe1998,Pe2000}.

\begin{lemma}(\cite{Ki2001})
\label{lem3}
Suppose $\varphi(x)=a(x)/b(x)$ for all $x\in dom \varphi$, where $a$ is a convex function, $b$ is finite and positive on $dom \varphi$, $dom \varphi$ is convex and one of the following conditions holds
\begin{enumerate}
\item[(a)] $b$ is affine;
\item[(b)] $a$ is nonnegative on $dom \varphi$ and $b$ is concave;
\item[(c)] $a$ is nonpositive on $dom \varphi$ and $b$ is convex.
\end{enumerate}
Then $\varphi$ is quasiconvex and $\partial (a-\alpha b) (x)$ is a subset of $\partial^{*}\varphi(x)$ for $\alpha=\frac{a(x)}{b(x)}$.
\end{lemma}



\section{ A Parallel Algorithm and Its Convergence}
  For presentation of the algorithm and its convergence, we make the following assumptions:

   \textbf{Assumptions}
\begin{enumerate}
\item[(A1)] For every $x\in C$, the function $f(x,.)$ is continuous, quasiconvex on $\R^n$, and $f(.,.)$ is upper semicontinuous on an open set containing $C\times C$;
\item[(A2)] The bifunction $f$ is pseudomonotone on $C$, that is
\begin{equation}
f(x,y)\geq 0 \Rightarrow f(y,x)\leq 0 \quad \forall x, y \in C,\nonumber
\end{equation}
and paramonotone on $C$ with respect to $S(EP)$, that is
\begin{equation}
x\in S(EP), y\in C \textnormal{ and } f(x,y)=f(y,x)=0 \Rightarrow y\in S(EP).\nonumber
\end{equation}
The paramonotonicity of a bifunction is an extension of that for an operator, see e.g., \cite{Ju1998}, which has been used in some papers, see, for example, \cite{AM2014,Sa2011,Yen2019}

\item[(A3)] The solution set $S(EP)$ is nonempty.


\end{enumerate}

 For simplicity of notation, for $\omega^T = (\omega_1,...,\omega_m)$ with $\omega_i > 0$ for every $i$ and $\sum_{i=1}^m \omega_i= 1$, let us define the operator $P_\omega$  as $P_\omega(x) := \sum_{i=1}^m \omega_i P_{C_i} (x)$ for every $x$. It is easy to see that $P_\omega$ is nonexpansive for any $\omega$. Since
 $C \subseteq C_i$ for every $i$, it follows that  $P_{C_i}(x) = P_\omega(x)=x$ for every  $x\in C$.
For this operator  we have the following result.

\begin{lemma}(\cite{De1998},\cite{DM2016})
\label{lem3.1D}
If $0<\omega_i<1$ for all $i=1,2,\dots,n$ and  $\bar{x}$ is a fixed poind of $P_\omega$, that is $\bar{x} = P_\omega(\bar{x})$, then $\bar{x} \in C = \cap_{i=1}^m C_i$.
\end{lemma}

  Now we are in a position to describe an algorithm for  solving equilibrium ($EP$), where the  feasible domain $C := \cap_{j=1}^m C_j$. The algorithm is a projection-subgradient one that takes the projection  independently on each component set.\\

\begin{algorithm}[H]\label{Alg1}
 Take real sequences $\{\alpha_k\}$, $\{\lambda_k\}$ and positive numbers $\omega_1,\dots,\omega_m$ satisfying  the following conditions
\begin{eqnarray}
 &\quad \alpha_k >0 \quad \forall k\in \N,&\nonumber\\
&\sum_{k=1}^{\infty} \alpha_k=+\infty, \quad \sum_{k=1}^{\infty} \alpha_k^2 <+\infty.&\nonumber
\\&0<\underline{\lambda}\leq \lambda_k \ \leq \overline{\lambda}<1 \quad \forall k\in \N,&\nonumber
\\&0<\omega_i<1 \quad \forall i=1,\dots,m, \sum_{i=1}^m \omega_i=1&\nonumber
\end{eqnarray}
\\\textbf{Initial Step:}  choose $x^0\in \R^n$, let $k=0$.
\\\textbf{Step k (0,1...):} Having $x^k \in \ \R^n$,
take $$g^k\in \partial_2^* f(x^k,x^k):=\{g\in \R^n :\langle g,y-x^k \rangle <0 \quad \forall y\in L_{f_k}(x^k)\}.$$

If $g^k = 0$ and $x^k\in C$,  \textbf{stop}: $x^k$ is a solution.

If $g^k \not=0$, normalize $g^k$ to obtain   $\|g^k\|=1.$

Compute
\begin{equation}
x^{k+1} =(1-\lambda_k)x^k+\lambda_k
P_{\omega}(x^k -\alpha_k g^k)
\end{equation}

If $x^{k+1}=x^k$ and $x^k\in C$ then \textbf{stop}: $x^k$ is a solution.

Else update $k\longleftarrow k+1$.
\caption{Subgradient-Projection Algorithm}
\end{algorithm}
\vskip0.5cm


 The sequence of the iterates generated by the algorithm has the following properties:

 \begin{proposition}\label{prop3}
 For every $z\in C=\cap_{i=1}^m C_i$, and $k\in \mathbb{N}$, the following inequality holds
 \begin{equation}
 \label{eq4.1}
 \|x^{k+1}-z\|^2 \leq \|x^k-z\|^2 +2\lambda_k \alpha_k \langle g^k, z-x^k\rangle +\lambda_k \alpha_k^2 -\lambda_k(1-\lambda_k) \|x^k - P_\omega (x^k -\alpha_k g^k)\|^2.
 \end{equation}\nonumber

 \end{proposition}
 \begin{proof}
 Let $z\in C$, by using the elementary equality
 $$ \|(1-t)a + tb\|^2 = (1-t)\|a\|^2 + t\|b\|^2 -t(1-t)\|a -b\|^2$$
 with $a=x^k-z$, $b= P_\omega(x^k-\alpha_k g^k) - z$, $t=\lambda_k$,
  we have
 \begin{eqnarray}
 \|x^{k+1}-z\|^2 &=& \|x^k +\lambda_k \Big[ P_\omega (x^k -\alpha_k g^k) -x^k\Big]-z\|^2\nonumber
 \\ &=& (1-\lambda_k) \|x^k -z\|^2 +\lambda_k\|P_\omega (x^k -\alpha_k g^k) -z\|^2\nonumber
 \\&& -\lambda_k(1-\lambda_k) \|x^k -P_\omega (x^k -\alpha_k g^k)\|^2.\label{eq4.3}
 \end{eqnarray}
 In addition,
 \begin{eqnarray}
 &&\|P_\omega (x^k -\alpha_k g^k) -z\|^2\nonumber
 \\ &=&\| \sum_{i=1}^m \omega_i(P_{C_i} (x^k -\alpha_k g^k)-z)\|^2 \nonumber
 \\ &\leq& \sum_{i=1}^m \omega_i \|P_{C_i} (x^k -\alpha_k g^k)-z\|^2\nonumber
 \\ &\leq&  \sum_{i=1}^m \omega_i \|x^k -\alpha_k g^k -z\|^2\nonumber
 \\ &=& \|x^k -\alpha_k g^k -z\|^2\nonumber
 \\ &= & \|x^k-z\|^2 -2\alpha_k \langle g^k, x^k-z\rangle +\alpha_k^2.\nonumber
 \end{eqnarray}
 From (\ref{eq4.3}) and the last equality, it follows that
 \begin{equation}
 \|x^{k+1}-z\|^2 \leq \|x^k-z\|^2 +2\lambda_k \alpha_k \langle g^k, z-x^k\rangle +\lambda_k \alpha_k^2 -\lambda_k(1-\lambda_k) \|x^k - P_\omega (x^k -\alpha_k g^k)\|^2.\nonumber
 \end{equation}
 \end{proof}

 \begin{lemma}
 \label{lem4.1}

 \begin{equation}
 \liminf_{k\rightarrow +\infty} \langle g^k, x^k -z\rangle \leq 0,  \  \forall z\in C.
 \end{equation}
  \end{lemma}
  \begin{proof}
  From Proposition \ref{prop3} and $0 <\underline{\lambda}\leq \lambda_k\leq \overline{\lambda}<1$, we obtain
  $$2\alpha_k \langle g^k, x^k-z\rangle  \leq \frac{1}{\underline{\lambda}}(\|x^k-z\|^2 -\|x^{k+1}-z\|^2) +\alpha_k^2.$$
  Applying this inequality for every $k=1,...\infty$, and summing up we obtain
  $$\sum_{k=1}^{\infty} \alpha_k \langle g^k, x^k-z\rangle <+\infty,$$
  which together with $\sum_{k=1}^{\infty} \alpha_k =+\infty$, implies
  \begin{equation}
 \liminf_{k\rightarrow +\infty} \langle g^k, x^k -z\rangle \leq 0.\nonumber
 \end{equation}
  \end{proof}
 We have the following convergence result.

  \begin{theorem}   Under the assumptions (A1 )- (A3) it holds that

(i) If Algorithm \ref{Alg1}  terminates at iteration $k$, then $x^k$ is a solution of ($EP$);

(ii) If the algorithm does not terminate, then there exists  a subsequence of $\{x^k\}$  converges to a solution of ($EP)$ whenever $\{x^k\}$ is bounded. In addition, if Problem($EP$) is uniquely solvable, in particular, $f$ is strongly monotone, the whole sequence $\{x^k\}$ converges to the   solution.
  \end{theorem}


  \begin{proof}

  (i) Suppose that the algorithm terminates at iteration $k$. Then, if $0\in \partial_2^* f(x^k,x^k)$ and $x^k\in C$, we have
$$x^k \in \textnormal{argmin}_{y\in \R^n} f(x^k,y).$$
Hence, $f(x^k,y)\geq f(x^k,x^k)=0$ for every $y\in \R^n$. Since $x^k\in C$,  it is a solution of ($EP$).

If $x^{k+1}=x^k$ and $x^k\in C$,  we have
\begin{eqnarray}
P_\omega (x^k -\alpha_k g^k) &=&x^k.\nonumber
\\ \Longleftrightarrow  \sum_{i=1}^m \omega_i P_{C_i}(x^k -\alpha_k g^k)&=&x^k. \nonumber
 \end{eqnarray}

 For every $y\in C\subseteq C_i$ and  every  $i\in \left\{1,\dots,m\right\}$, it holds that
 \begin{equation}
 \langle x^k -\alpha_kg^k -  P_{C_i}(x^k -\alpha_k g^k), y - P_{C_i}(x^k -\alpha_k g^k)\rangle \leq 0.\nonumber
\end{equation}
Equivalently
 \begin{equation}
 \langle x^k -  P_{C_i}(x^k - \alpha_k g^k), y - P_{C_i}(x^k -\alpha_k g^k)\rangle \leq \langle \alpha_k g^k, y - P_{C_i}(x^k -\alpha_k g^k)\rangle.\nonumber
\end{equation}

Multiplying by $\omega_i$ and summing up we obtain

 \begin{equation}
 \sum_{i=1}^m \omega_i \langle x^k -  P_{C_i}(x^k - \alpha_k g^k), y - P_{C_i}(x^k -\alpha_k g^k)\rangle \leq \sum_{i=1}^m \omega_i \langle \alpha_k g^k, y - P_{C_i}(x^k -\alpha_k g^k)\rangle.\nonumber
\end{equation}

By a simple computation, using $\sum_{i=1}^m \omega_i P_{C_i}(x^k -\alpha_k g^k)=x^k$ and $\sum_{i=1}^m \omega_i =1$, we arrive at

\begin{equation}
- \|x^k\|^2 + \|x^k\|^2 \leq \sum_{i=1}^m \omega_i \alpha_k\langle g^k, y-P_{C_i} (x^k - \alpha_k g^k)\rangle = \alpha_k
\langle g^k, y-x^k\rangle. \nonumber
\end{equation}
Since $g^k\in \partial^*f(x^k,x^k)$, the last inequality $\langle g^k, y-x^k\rangle \geq 0, \ \forall y\in C$ implies $f(x^k,y) \geq f(x^k,x^k) = 0$ for every $y\in C$, which means that $x^k$ is a solution of (EP).

(ii)  We consider two cases.

  \textbf{Case 1:} There exists a  solution $x^*\in S(EP)$ and an index $k_0$ such that for $k\geq k_0$,
  $$\|x^{k+1} -x^*\| \leq \|x^k -x^*\|.$$
  By the nonnegativity of $\|x^k -x^*\|$, we conclude that the sequence $\left\{\|x^k -x^*\|\right\}$ is convergent.

 Moreover, from Proposition \ref{prop3} it follows that
 \begin{eqnarray}
 &&\lambda_k(1-\lambda_k) \|x^k - P_\omega (x^k -\alpha_k g^k)\|^2 \nonumber\\&\leq& \|x^k -x^*\|^2 -\|x^{k+1}-x^*\|^2 -2\lambda_k \alpha_k \langle g^k, x^k -x^*\rangle +\lambda_k\alpha_k^2.\nonumber
 \end{eqnarray}
 When $k$ goes to infinity, the right hand side of the above inequality goes to $0$.  Since $0 <\underline{\lambda}\leq \lambda_k\leq \overline{\lambda}<1$,  we obtain
 $$\lim_{k \rightarrow \infty} \|x^k - P_\omega (x^k -\alpha_k g^k)\| =0.$$

 We also have that
 \begin{eqnarray}
 \|x^k -P_\omega (x^k)\| &\leq& \|x^k - P_\omega (x^k -\alpha_k g^k)\| + \|P_\omega (x^k -\alpha_k g^k)- P_\omega (x^k)\|\nonumber
 \\&\leq & \|x^k - P_\omega (x^k -\alpha_k g^k)\| + \| x^k  -\alpha_k g^k - x^k\|\nonumber
 \\&\leq & \|x^k - P_\omega (x^k -\alpha_k g^k)\| +\alpha_k.
 \end{eqnarray}
 Since $\lim_{k \rightarrow \infty} \alpha_k =0$,
\begin{equation}\lim_{k \rightarrow \infty} \|x^k - P_\omega (x^k)\|=0.\label{eq4.5}
\end{equation}
From Lemma \ref{lem4.1} follows
\begin{equation}
 \liminf_{k\rightarrow \infty} \langle g^k, x^k -x^*\rangle \leq 0.\label{eq4.8}
 \end{equation}

 Let $\left\{x^{k_i}\right\}$ be a subsequence of $\left\{x^k\right\}$ such that
 $$\lim_{i\rightarrow \infty} \langle g^{k_i}, x^{k_i} -x^*\rangle = \liminf_{k\rightarrow \infty} \langle g^k, x^k -z\rangle.$$
 Since $\{x^k\}$ is bounded,  $\{x^{k_i}\}$ is bounded too. Let $\overline{x}$ be a limit point of $\{x^{k_i}\}$, and without loss of generality we assume that
 \begin{equation}
 \lim_{i\rightarrow \infty} x^{k_i}=\overline{x}.\label{eq4.6}
 \end{equation}
It is clear that
\begin{eqnarray}
\|x^{k_i}-P_\omega (\overline{x}\|&\leq & \|x^{k_i} -P_\omega (x^{k_i})\|+\|P_\omega (\overline{x}) -P_\omega (x^{k_i})\|\nonumber
\\&\leq & \|x^{k_i} -P_\omega (x^{k_i})\| +\|\overline{x}- x^{k_i}\|.
\end{eqnarray}
Thanks to (\ref{eq4.5}) and (\ref{eq4.6}), we have
 \begin{equation}
 \lim_{i\rightarrow \infty} x^{k_i}=P_\omega (\overline{x}).\label{eq4.7}
 \end{equation}
 Combining this with (\ref{eq4.6}), we  see that $P_\omega (\overline{x})=\overline{x}$, which together with
  $\omega_i>0$ for every $i$ implies  $\overline{x}\in C$.

In addition, since $x^*$ is a solution, by pseudomonotonicity of $f$ on $C$, we have
$f(\overline{x}, x^*)\leq 0.$  We  show that $f(\overline{x}, x^*)= 0.$ In fact,
by contradiction, we assume that there exists $a>0$ such that
$$f(\overline{x}, x^*)\leq -a.$$
 Since $f(., .)$ is upper semicontinuous on an open set containing $C\times C$, there exist positive numbers $\epsilon_1, \epsilon_2$ such that,  for any $x\in B(\overline{x}, \epsilon_1)$, $y\in B(x^*, \epsilon_2)$ we have
$f(x, y)\leq -\frac{a}{2}.$

On the other hand, $\lim_{i\rightarrow \infty} x^{k_i}=\overline{x}$ implies that there exist $i_0$ such that for $i\geq i_0$, $x^{k_i}$ belongs to $B(\overline{x}, \epsilon_1)$.
So, for $i\geq i_0$ and $y\in B(x^*, \epsilon_2)$, we have
\begin{equation}\label{them}
f(x^{k_i}, y)\leq -\frac{a}{2},
\end{equation}
 which implies that $B(x^{*}, \epsilon_2)\subset  L_{f_{k_i}}(x^{k_i})$. In addition, $g^{k_i}\not=0$, because if $g^{k_i}=0$ then $f(x^{k_i},y)\geq f(x^{k_i},x^{k_i})=0$ for every $y\in \mathbb{R}^n$, which contradicts to (\ref{them}). By Lemma \ref{lem3.1a},   for $i\geq i_0$, it holds that
\begin{equation}
\langle g^{k_i}, x^{k_i}-x^{*}\rangle >\epsilon_2,\nonumber
\end{equation}
 which  contradicts  to (\ref{eq4.8}). Thus, $\overline{x}\in C$ and $f(\overline{x}, x^*)=0$.
 Again by pseudomonotonicity, we obtain $f(x^*,\overline{x}) = 0.$  Then, from paramonotonicity of $f$ it follows that $\overline{x}$ is a solution of $(EP)$.
In addition, since $\{\|x^k-x^*\|\}$ is convergent for every solution $x^*$ of $(EP)$, the sequence $\{x^k\}$ converges to $\overline{x}$ which is a solution of $(EP)$.

 \vskip0.5cm
 \textbf{Case 2:} For any solution $x^*$ of (EP)  there exists a subsequence $\{x_{k_i}\}$ of $\{x^k\}$ that satisfies
 $$\|x^{k_i}-x^*\| <\|x^{k_i+1}-x^*\|.$$


 Now again, by Proposition \ref{prop3} applying with $x^*$, we can write
 \begin{eqnarray}
 &&\lambda_{k_i}(1-\lambda_{k_i}) \|x^{k_i} - P_\omega (x^{k_i} -\alpha_{k_i} g^{k_i})\|^2 \nonumber\\&\leq& \|x^{k_i} -x^*\|^2 -\|x^{{k_i}+1}-x^*\|^2 -2\lambda_{k_i} \alpha_{k_i} \langle g^{k_i}, x^{k_i} -x^*\rangle +\lambda_{k_i}\alpha_{k_i}^2.\nonumber
 \\&\leq& -2\lambda_{k_i} \alpha_{k_i} \langle g^{k_i}, x^{k_i} -x^*\rangle +\lambda_{k_i}\alpha_{k_i}^2.\label{eq4.9}
 \end{eqnarray}

 Since $\{x^{k_i}\}$ is bounded, and $0 < \lambda_k < 1, $ $\alpha_k \to 0$, from the last inequality it follows  that
 \begin{equation}
 \lim_{i\rightarrow \infty} \|x^{k_i} - P_\omega (x^{k_i} -\alpha_{k_i} g^{k_i})\| =0.
 \end{equation}

 Also from (\ref{eq4.9})  follows
 $$  \langle g^{k_i}, x^{k_i} -x^*\rangle\leq \frac{\alpha_{k_i}}{2}.$$
\begin{equation}
 \limsup_{i\rightarrow \infty} \langle g^{k_i}, x^{k_i} -x^*\rangle \leq 0.\label{eq4.10}
 \end{equation}
 Now by the same argument as in Case 1, we see that every limit point of the sequence $\{x_{k_i}\}$ belongs to the solution set $S(EP)$.

 Suppose now that   the solution   is unique, then the sequence $\{x^k\}$ converges to the unique solution $x^*$ of $(EP)$. In fact, without loss of generality, we may assume that for $k\not\in \{k_j\}$. Then
 $$\|x^{k+1}-x^*\| \leq \|x^k-x^*\|.$$
 For any $\epsilon>0$ there exist $i_0$ such that for $i\geq i_0$,
 $$\|x^{k_i}-x^*\|<\epsilon.$$
 Furthermore, for any $k> k_{i_0}$ that does not belong to $\{k_i\}$, there exists $i_1\geq i_0$ such that $k_{i_1}<k<k_{i_1+1}$. Then
 $\|x^k-x^*\| \leq \|x^{k_{i_1}}-x^*\|<\epsilon,$
 which means  that
 $\|x^k -x^* \|<\epsilon$ for $k\geq k_{i_0}$.
 Hence,   the whole sequence $\{x^k\}$ converges to the unique solution of $(EP)$.

  \end{proof}
 {\bf Remark}.
The assumption on boundedness of the sequence $\{x^k\}$ is ensured  if either

   (C1) The set  $ C_\omega := \{ x\in \R^n: \ x = \sum_{j=1}^m \omega_j x^j: \ x^j \in C_j \ j=1,...,m \}$ is bounded, and $x^0 \in C_\omega$,

    or

 (C2)   The bifuntion $f(x,y)$ is pseudomonotone on $C_\omega$ and for each $x$,  the function $f_x(.):= f(x,.)$ is Lipschitz continuous  with constant $L_x$, and
$$C\ \cap S(C_\omega, f) \not=\emptyset,$$
 where $S(C_\omega,f) $ stands for the solution-set of the  equilibrium problem
$$\textnormal {find}\ \bar{z} \in C_\omega: \ f( \bar{z}, z) \geq 0\  \forall z\in C_\omega. \eqno(EP_\omega)$$

Indeed, according to  the algorithm, $x^{k+1} $ is a convex combination of elements of the convex set $C_\omega$,  we see that $x^{k+1} \in C_\omega$ for every $k=0, 1, ...,$, with implies that $\{x^k\}$ is bounded. \\
To see the assertion for (C2), we apply Lemma \ref{Ko1} with $h(x):= f(x^k,x)$,  $z:= x^k$ and $y=\bar{z} \in C\ \cap S(C_\omega, f)$. Then  $f(y,x^k)\geq 0$, which, by pseudomonotonicity,  implies $f(x^k, y) \leq 0$ for all $k$. Hence $y\in \bar{L}_h(x^k)$. Thus by  Lemma \ref{Ko1}, $\langle g^k, \bar{z} - x^k\rangle \leq 0$. Then, by applying Proposition \ref{prop3} with $z = \bar{z}$, we  obtain
   $$ \|x^{k+1}-\bar{z}\|^2 \leq \|x^k - \bar{z}\|^2 + \alpha^2_k$$
   Hence, $\{\|x^k - \bar{z}\|^2\}$ is convergence, and therefore $\{x^k\}$ is bounded

  In the case of optimization problem, where $f(x,y) := \varphi(y) -\varphi (x)$, Condition (C2) means that
 $\varphi$ is Lipschitz continuous
  and there exists $\bar{z}\in C$ such that $\bar{z} $ is  a minimizer of the function $\varphi(y)$ on $C_\omega$.

 Condition $(C2)$ is inspired by the paper  \cite{San2016},  where the equilibrium  problems of the form
$$\textnormal{Find} \ x^* \in P\cap Q: f(x^*,y) \geq 0 \  \forall  y\in Q,$$
 with $P, Q$ being convex sets have been studied. This common solution problem has been  attracted  much attention of researchers in recent years.

 \section{Computational Experience}
  To test the algorithm, we consider the  equilibrium problem
$$
\textnormal{Find } x^*\in C \textnormal{ such that } f(x^*,y) \geq 0 \quad \forall y\in C, \eqno (EP)
$$
where the bifunction $f(x,y)$ is defined by
\begin{equation}
f(x,y)=\left\langle Ax+b, \frac{A_1y+b_1}{c^Ty+d}-\frac{A_1x+b_1}{c^Tx+d}\right\rangle,
\end{equation}
with $A,A_1\in \R^{n\times n}$, $b,b_1,c\in \R^n$, $d\in \R$ and $C\subset \{x| \quad c^Tx+d >0\}$.

By applying  Proposition 4.1 in \cite{Ju1998} to the differentiable function
   $ g(x):= A^T (\frac{A_1x+b_1}{c^Tx + d})$
   one can easily  check that $f$ is monotone on $C$
if and only if the matrix
$$\hat{A}_1 (x)=A^T\Big[ A_1c^Tx-A_1x c^T\Big] + A^T\Big[A_1d -b_1c^T\Big]$$
is  positive semidefinite and paramonotone if $\hat{A}_1 (x)$ is symmetric for any $x\in C$.

 We test the algorithm with the following three examples, where the projection onto each component set has a closed form. For each problem we chose $\lambda_k = 1/2$ for every $k$, and $\omega_i = 1/m$ for all $i$.
We stop the computation at iteration
$k$ if $g^k = 0$ or $err1 :=\| x^k - x^{k+1}\|   < 10^{-4}$ and $err2 = \|x^k -   P_{C_1} (x^k)\| + ...+ \|x^k - P_{C_m} (x^k) \|  <10^{-1}$,  or the number of iteration exceeds 1000.

The average time and average errors for each size are reported in Tables  1, 2, 3 with different sizes, a hundred of  problems have been tested for each size.

\begin{example}
In this example, we take
$$C=C_1\cap C_2,$$
where  $C_1=\left[ 1,3 \right]^n$ and $C_2=\{x\in \mathbb{R}^n| \quad \|x\|\leq 3\}$. Each entries of $A,A_1,b,b_1,c,d$ is uniformly generated in the interval $\left[0,1\right]$.
\begin{table}[ht]
\caption{Algorithm with $\alpha_k=\frac{100}{k+1}$} 
\centering 
\begin{tabular}{c c c c c c c} 
\hline\hline 
 n &N. of prob.   & CPU-times(s)& Error1& Error 2\\ [0.5ex] 
\hline 
5&100& 5.674775& 0.000127& 0.188295& \\
10&100&  6.990620&0.000092&0.187407& \\
20&100& 8.7481101& 0.000085& 0.186929&   \\
50&100& 40.770369&0.000083& 0.186708&  \\
\hline 
\end{tabular}
\label{tab1} 
\end{table}
\end{example}
\begin{example}
In this example, we take
$$C=C_1\cap C_2\cap C_3,$$
where  $C_1=\left[ 1,3 \right]^n$, $C_2=\{x\in \mathbb{R}^n| \quad \|x\|\leq 3\}$ and $C_3=\{x\in \mathbb{R}^n| \quad \sum_{i=1}^n x_i \geq N+1\}$. Each entries of $A,A_1,b,b_1,c,d$ is uniformly generated in the interval $\left[0,1\right]$.

\begin{table}[ht]
\caption{Algorithm with $\alpha_k=\frac{100}{k+1}$} 
\centering 
\begin{tabular}{c c c c c c c c} 
\hline\hline 
 n &N. of prob. & CPU-times(s)& Error 1& Error 2\\ [0.5ex] 
\hline 
5&100& 11.339168& 0.000224& 0.383520& \\
10&100&  11.638716& 0.000182& 0.383298& \\
20&100& 13.612659& 0.000170& 0.386529&   \\
50&100& 42.999559&0.000165& 0.388552&  \\
\hline 
\end{tabular}
\label{tab2} 
\end{table}
\end{example}
\begin{example}
In this example, we take
$$C=C_1\cap C_2,$$
where  $C_1=\left[ 1,3 \right]^n$ and $C_2=\{x\in \mathbb{R}^n| \quad \sum_{i=1}^3 x_i \geq 3\}$.  Each entries of $A,A_1,b,b_1,c,d$ is uniformly generated in the interval $\left[0,1\right]$.

Note that $x^*\in C$ is a solution of the equilibrium problem (EP) if and only if $x^*$ belongs to the solution set of the following affine fractional programming problem
$$\min_{y\in C} g(x^*,y),$$
where $ g(x,y)=\langle Ax+b, \frac{A_1y+b_1}{c^Ty+d}\rangle$.
Thus  we can use linear programming algorithms to compute
$$err_3=\frac{-\min_{y\in C} g(P_C(x^k),y)+g(P_C(x^k), P_C(x^k))}{g(P_C(x^k), P_C(x^k))}.$$
and use it as a stopping criterion.
\begin{table}[ht]
\caption{Algorithm with $\alpha_k=\frac{100}{k+1}$} 
\centering 
\begin{tabular}{c c c c c c c c} 
\hline\hline 
 n &N. of prob.   & CPU-times(s)& Error1& Error 2& Error 3\\ [0.5ex] 
\hline 
5&100& 10.218638& 0.000210& 0.199647&0.058665&\\
10&100&  10.811854&0.000275&0.200397&0.051829&\\
20&100& 12.970668& 0.000498& 0.200389&0.074603&   \\
50&100& 45.202791&0.000731& 0.200382& 0.086886& \\
\hline 
\end{tabular}
\label{tab3} 
\end{table}
\end{example}

\section{Conclustion}
We have proposed an iterative  star-subgradient  projection algorithm for solving a class of equilibrium problems over the intersection of closed, convex sets, where the bifunction is quasiconvex in its second variable.  The search direction at each iteration is defined by a star-subgradient at the current iterate, and the projection is executed independently  on each component  of the intersection sets.  Convergence of the algorithm has been shown, and some illustrative examples have been solved.


\end{document}